\documentclass[12pt]{amsart}

\usepackage{amsmath}
\usepackage{amssymb}
\usepackage{amsfonts}
\usepackage{amsthm}
\usepackage{enumerate}
\usepackage{hyperref}
\usepackage{color}

\theoremstyle{plain}
\newtheorem{thm}{Theorem}[section]

\newtheorem{prop}[thm]{Proposition}

\newtheorem{cor}[thm]{Corollary}

\newtheorem{ques}[thm]{Question}

\newtheorem{conj}[thm]{Conjecture}

\theoremstyle{definition}
\newtheorem{dfn}[thm]{Definition}

\newtheorem{rem}[thm]{Remark}

\newtheorem{dfns-rems}[thm]{Definitions and Remarks}
\newtheorem{notas-rems}[thm]{Notations and Remarks}
\newtheorem{exmps-rems}[thm]{Examples and Remarks}

\begin{document}


\title[Free resolution of powers of monomial ideals and Golod rings]{Free resolution of powers of monomial ideals and Golod rings}


\author[N. Altafi]{N. Altafi}

\address{N. Altafi, School of Mathematics, Statistics and Computer Science,
College of Science, University of Tehran, Tehran, Iran, and School of
Mathematics, Tehran, Iran.}

\email{nasrin.altafi@gmail.com}

\author[N. Nemati]{N. Nemati}

\address{N. Nemati, School of Mathematics, Statistics and Computer Science,
College of Science, University of Tehran, Tehran, Iran, and School of
Mathematics, Tehran, Iran.}

\email{navid.nemati.math@gmail.com}

\author[S. A. Seyed Fakhari]{S. A. Seyed Fakhari}

\address{S. A. Seyed Fakhari, School of Mathematics, Institute for Research
in Fundamental Sciences (IPM), P.O. Box 19395-5746, Tehran, Iran.}

\email{fakhari@ipm.ir}

\urladdr{http://math.ipm.ac.ir/fakhari/}

\author[S. Yassemi]{S. Yassemi}

\address{S. Yassemi, School of Mathematics, Statistics and Computer Science,
College of Science, University of Tehran, Tehran, Iran, and School of
Mathematics, Institute for Research in Fundamental Sciences (IPM), P.O. Box
19395-5746, Tehran, Iran.}

\email{yassemi@ipm.ir}

\urladdr{http://math.ipm.ac.ir/yassemi/}


\begin{abstract}
Let $S = \mathbb{K}[x_1, \dots, x_n]$ be the polynomial ring over a field $\mathbb{K}$. In this paper we present a criterion for componentwise linearity of powers of monomial ideals. In particular, we prove that if a square-free monomial ideal $I$ contains no variable and some power of $I$ is componentwise linear, then $I$ satisfies gcd condition. For a square-free monomial ideal $I$ which contains no variable, we show that $S/I$ is a Golod ring provided that for some integer $s\geq 1$, the ideal $I^s$ has linear quotient with respect to a monomial order. We also provide a lower bound for some Betti numbers of powers of a square-free monomial ideal which is generated in a single degree.
\end{abstract}


\subjclass[2000]{Primary: 13D02; Secondary: 13F20 13C14}


\keywords{Monomial ideal, Componentwise linear, Betti number, {\rm Lcm} lattice, Golod ring, Linear quotient}


\thanks{The research of S. A. Seyed Fakhari, and S. Yassemi
was in part supported by a grant from IPM (No. 92130422, and
No. 92130214)}


\maketitle


\section{Introduction and preliminaries} \label{sec1}

Over the last 20 years the study of algebraic, homological and combinatorial
properties of powers of ideals has been one of the major topics in Commutative
Algebra. In this paper we study the minimal free resolution of the powers of monomial ideals. First we give some definitions and basic facts.

Let $S = \mathbb{K}[x_1, \dots, x_n]$ be the polynomial ring over a field $\mathbb{K}$. For any
finitely generated $\mathbb{Z}^n$-graded $S$-module $M$ and every $\mathbf{a}\in \mathbb{Z}^n$, let $M_{\mathbf{a}}$ denote its graded piece of degree $\mathbf{a}$
and let $M(\mathbf{a})$ denote the {\it twisting} of $M$ by $\mathbf{a}$, i.e. the module where $M(\mathbf{a})_{\mathbf{b}}= M_{\mathbf{a}+\mathbf{b}}$. As usual for every vector $\mathbf{a}=(a_1, \ldots, a_n)\in \mathbb{Z}^n$, we denote by $\mid \mathbf{a}\mid$ the absolute value of $\mathbf{a}$ which is equal to $a_1+ \ldots +a_n$.

A {\it minimal graded free resolution} of $M$ is an exact complex
$$0 \longrightarrow F_p \longrightarrow F_{p-1} \longrightarrow  F_1 \longrightarrow  F_0 \longrightarrow  M \longrightarrow 0,$$
where each $F_i$ is a $\mathbb{Z}^n$-graded free $S$-module of the form $\oplus_{\mathbf{a}\in \mathbb{Z}^n}S(-\mathbf{a})^{\beta_{i,\mathbf{a}}(M)}$ such that the
number of basis elements is minimal and each map is graded.
The value $\beta_{i,\mathbf{a}}(M)$ is called the $i$th $\mathbb{Z}^n$-graded Betti number of degree $\mathbf{a}$. The number $\beta_{i,j}(M)=\sum_{\mid \mathbf{a}\mid=j}\beta_{i,\mathbf{a}}(M)$ is called the $i$th $\mathbb{Z}$-graded Betti number of degree $j$. These numbers are independent
of the choice of resolution of $M$, thus yielding important numerical invariants of
$M$. To simplify the notations, for every monomial $u=\mathbf{x^a}=x_1^{a_1}\ldots x_n^{a_n}$, we write $\beta_{i,\mathbf{u}}(M)$ instead of $\beta_{i,\mathbf{a}}(M)$.

Let $I$ be a Monomial ideal of $S$ and let $G(I)$ denote the set of minimal monomial generators of $I$. Then $I$ is said to
have {\it linear resolution}, if for some integer $d$, $\beta_{i,i+t}(I)=0$
for all $i$ and every $t\neq d$. It is clear from the definition that if a monomial ideal has a linear resolution, then all the minimal monomial generators of $I$ have the same degree. There are many attempts to characterize the monomial ideals with linear resolution. One of the most important results in this direction is due to Fr${\rm \ddot{o}}$berg \cite[Theorem 1]{f}, who characterized all square-free monomial ideals generated by quadratic monomials, which have linear resolutions. It is also interesting to ask whether some powers of a given monomial ideal $I$ has linear resolution. It is known \cite{ht} that polymatroidal ideals have linear resolutions and that powers
of polymatroidal ideals are again polymatroidal (see \cite{hh}). In particular
they have again linear resolutions. In general however, powers of ideals with linear
resolution need not to have linear resolutions. The first example of such an ideal
was given by Terai. He showed that over a base field of characteristic $\neq 2$ the Stanley
Reisner ideal $I = (x_1x_2x_3, x_1x_2x_5, x_1x_3x_6, x_1x_4x_5, x_1x_4x_6, x_2x_3x_4, x_2x_4x_6, x_2x_5x_6, x_3x_4x_5, x_3x_5x_)$ of the minimal
triangulation of the projective plane has a linear resolution, while $I^2$ has no linear
resolution. This example depends on the characteristic of the base field. If the base
field has characteristic $2$, then I itself has no linear resolution.
Another example, namely $I = (x_4x_5x_6, x_3x_5x_6, x_3x_4x_6, x_3x_4x_5, x_2x_5x_6, x_2x_3x_4, x_1x_3x_6, x_1x_4x_5)$ is given by
Sturmfels \cite{s}. Again $I$ has a linear resolution, while $I^2$ has no linear resolution. However, Herzog, Hibi and Zheng \cite{hhz} prove that a monomial ideal $I$ generated in degree $2$ has linear resolution if and only if every power of $I$ has linear resolution.

Componentwise linear ideals are introduced by Herzog and Hibi \cite{hh2} and they are defined as follows.
For a monomial ideal $I$ we write $I_{\langle j\rangle}$
for the ideal generated by all monomials of degree $j$
belonging to $I$. A monomial ideal $I$ is called {\it
componentwise linear} if $I_{\langle j\rangle}$ has a linear resolution for
all $j$. In Section \ref{sec2} we study the componentwise linearity of powers of monomial ideals. More explicit, we prove that if $I$ is a monomial ideal, which contains no variable and some power of $I$ is componentwise linear, then for every couple of monomials $u,v\in G(I)$ with ${\rm gcd}(u,v)=1$, there exists a monomial $w\in G(I)$ such that $w\neq u,v$ and ${\rm supp}(w)\subseteq {\rm supp}(u)\cup {\rm supp}(v)$ (see Theorem \ref{main}). Let $I$ be a square-free monomial ideal which contains no variable. We prove that if some power of $I$ has linear quotient with respect to a monomial order (see Definition \ref{lq}), then $S/I$ is a Golod ring (see Theorem \ref{golod}).  One of the main tool which is used in this paper is the lcm lattice, defined by Gasharov, Peeva and Welker \cite{gpw}.

{\bf Lcm lattice. }Let $I$ be a monomial ideal minimally generated by monomials
$m_1,\ldots,m_d$. We denote by $L_I$ the lattice with elements labeled by
the least common multiples of $m_1,\ldots,m_d$ ordered by divisibility. In
particular, the atoms in $L_I$ are $m_1,\ldots, m_d$, the maximal element
is ${\rm lcm}(m_1,\ldots,m_d)$, and the minimal element is $1$ regarded as
the lcm of the empty set. The least common multiple of elements in $L_I$ is
their join, i.e., their least common upper bound in the poset $L_I$. We
call $L_I$ the {\it lcm-lattice} of $I$. Gasharov, Peeva and Welker have
proven that the lcm-lattice of a monomial ideal determines its free
resolution. For an open interval $(1,m)$
in $L_I$, we denote by $\Delta(1,m)$ the order complex of the interval, i.e., the simplicial complex whose facets are the maximal chains in $(1,m)$. Let
$\widetilde{H}_i(\Delta(1,m); \mathbb{K})$ denote the $i$th reduced homology of $\Delta(1,m)$ with coefficients
in $\mathbb{K}$. By \cite[Theorem 3.3]{gpw}, the $\mathbb{Z}^n$-graded Betti numbers of $I$ can be computed
by the homology of the open intervals in $L_I$ as follows: if $m \notin L_I$ then
$\beta_{i,m}(I) = 0$ for every $i$, and if $m \in L_I$ and $i\geq 1$ we have $$\beta_{i,m}(I)={\rm dim}_{\mathbb{K}}\widetilde{H}_{i-1}(\Delta(1,m); \mathbb{K}).$$Using this formula in Section \ref {sec3}, we provide a lower bound for some Betti numbers of powers of a square-free monomial ideal which is generated in a single degree (see Theorem \ref{beti}).


\section{Componentwise linearity and Golod rings} \label{sec2}

The following theorem is the first main result of this paper.

\begin{thm} \label{main}
Let $I$ be a monomial ideal, which contains no variable. Assume that there exists an integer $s\geq 1$ such that $I^s$ is componentwise linear. Then for every couple of monomials $u,v\in G(I)$ with ${\rm gcd}(u,v)=1$, there exists a monomial $w\in G(I)$ such that $w\neq u,v$ and ${\rm supp}(w)\subseteq {\rm supp}(u)\cup {\rm supp}(v)$.
\end{thm}

\begin{proof}
By contradiction, suppose that there exist $u,v\in G(I)$ such that ${\rm gcd}(u,v)=1$ and there is no monomial $w\in G(I)$ such that ${\rm supp}(w)\subseteq {\rm supp}(u)\cup {\rm supp}(v)$. Without loss of generality assume that ${\rm deg}(u)=d'\leq {\rm deg}(v)=d$. Note that $v^s\in I^s_{\langle ds \rangle}$. On the other hand $uv^{s-1}\in I^s$. So if we multiply $uv^{s-1}$ to a divisor of $v$ of degree $d-d'$, we obtain a monomial $u'\in I^s_{\langle ds \rangle}$ with ${\rm lcm}(u', v^s)=uv^s$.

Consider the open interval $(1, uv^s)$ in the lcm lattice of $I^s_{\langle ds \rangle}$. We claim that the atom $v^s$ is an isolated vertex of $\Delta(1, uv^s)$. Assume that this is not true. Then there exists an atom $w'\in (1, uv^s)$ such that $w'\neq v^s$ and ${\rm lcm}(v^s, w')$ strictly divides $uv^s$. This implies that ${\rm gcd}(w', u)\neq 1$. Since $w'\in I^s$ and since there is no monomial $w\in G(I)$ such that ${\rm supp}(w)\subseteq {\rm supp}(u)\cup {\rm supp}(v)$, it follows that $u\mid w'$. Thus ${\rm lcm}(v^s, w')=uv^s$, which is a contradiction. This proves our claim. Now $u'$ is another vertex of $\Delta(1, uv^s)$ and thus $\Delta(1, uv^s)$ is disconnected. Hence by \cite[Theorem 3.3]{gpw}
\begin{center}
$\beta_{1,uv^s}(I^s_{\langle ds \rangle})={\rm dim}_{\mathbb{K}}\tilde{H}_0(\Delta(1,uv^s);\mathbb{K})\geq1$
\end{center}
This in particular shows that $\beta_{1, ds+d'}(I^s_{\langle ds \rangle})\neq 0$. Since $I$ contains no variable, $d'\geq 2$ and therefore the minimal free resolution of $I^s_{\langle ds \rangle}$ is not linear. Which is a contradiction.
\end{proof}

Monomial ideals with gcd condition and strong gcd condition are defined by J${\rm \ddot{o}}$llenbeck as follows.

\begin{dfn}
\cite[Definition 3.8]{j}
\begin{itemize}
\item[(i)] A monomial ideal $I$ is said to satisfy the {\it gcd condition}, if for any two monomials $u,v\in G(I)$
with ${\rm gcd}(u, v) = 1$ there exists a monomial $w\neq u,v$ in $G(I)$ with $w\mid {\rm lcm}(u, v)=uv$.

\item[(ii)] A monomial ideal $I$ is said to satisfy the {\it strong gcd condition}, if there exists a linear order $\prec$ on
$G(I)$ such that for any two monomials $u\prec v\in G(I)$
with ${\rm gcd}(u, v) = 1$ there exists a monomial $w\neq u,v$ in $G(I)$ with $u\prec w$ and $w\mid {\rm lcm}(u, v)=uv$.
\end{itemize}
\end{dfn}

One should note that a square-free monomial ideal $I$ satisfies gcd condition if and only if for every couple of monomials $u,v\in G(I)$ with ${\rm gcd}(u,v)=1$, there exists a monomial $w\in G(I)$ such that $w\neq u,v$ and ${\rm supp}(w)\subseteq {\rm supp}(u)\cup {\rm supp}(v)$. Thus as a consequence of Theorem \ref{main} we conclude the following corollary.

\begin{cor} \label{scomp}
Let $I$ be a square-free monomial ideal which contains no variable. Assume that $I^s$ is componentwise linear for some $s\geq 1$, then $I$ satisfies gcd condition.
\end{cor}

To any finite simple graph $G$ with vertex set
$V(G)=\{v_1,\ldots,v_n\}$ and edge set $E(G)$, one associates an ideal
$I(G)$ generated by all quadratic monomials $x_ix_j$
such that $\{v_i,v_j\}\in E(G)$. The ideal $I(G)$ is called the {\it edge ideal of $G$}.
We recall that for a graph $G=(V(G),E(G))$, its {\it complementary graph}
$G^c$ is a graph with $V(G^c)=V(G)$ and $E(G^c)$
consists of those $2$-element subsets $\{v_i,v_j\}$ of $V(G)$ for which
$\{v_i,v_j\}\notin E(G)$. If we restrict ourselves to edge ideal of graphs, we obtain the following result which was proved by Nevo and Peeva \cite[Proposition 1.8]{np}.

\begin{cor} \label{graph}
Let $I=I(G)$ be the edge ideal of a graph $G$. If $I^s$ has linear resolution for some $s\geq 1$, then $G^c$ has no induced $4$-cycle.
\end{cor}
\begin{proof}
Assume that the assertion is not true. So $G^c$ has an induced $4$-cycle, say $v_1, v_2, v_3, v_4$. Set $u=x_1x_3$ and $v=x_2x_4$. Then $u, v\in G(I)$ and ${\rm gcd}(u,v)=1$. But there is no monomial $w\in G(I)$ such that $w\mid uv$. Hence it follows from Corollary \ref{scomp} that no power of $I$ can have linear resolution.
\end{proof}
\begin{rem}
Corollary \ref{graph} is due to Francisco-H${\rm \grave{a}}$- Van Tuyl (non-published work). A short proof is presented by
Nevo and Peeva in \cite[Proposition 1.8]{np}; note that there is a repeated typo in their proof and $(x_px_q)^s$ has to be
replaced by $x_px_q$ throughout that proof.
\end{rem}

Ideals with linear quotients were first considered in \cite{ht} and they are a large subclass of componentwise linear ideals.

\begin{dfn} \label{lq}
Let $I$ be a monomial ideal and let $G(I)$ be the set of minimal monomial generators of $I$. Assume that $u_1\prec u_2 \prec \ldots \prec u_t$ is a linear order on $G(I)$. We say that $I$ has {\it linear quotient with respect to $\prec$}, if for every $2\leq i\leq t$, the ideal $(u_1, \ldots, u_{i-1}):u_i$ is generated by a subset of variables. We say that $I$ has {\it linear quotient}, if it has linear quotient with respect to a linear order on $G(I)$.
\end{dfn}

In the following proposition we examine linear quotient for powers of monomial ideals.

\begin{prop} \label{strong}
Let $I$ be a monomial ideal which contains no variable and let $s\geq 1$ be an integer. Assume that $I^s$ has linear quotient with respect to a monomial order $<$ on $G(I^s)$. Then there exists a linear order $\prec$ on
$G(I)$ such that for any two monomials $u\prec v\in G(I)$
with ${\rm gcd}(u, v) = 1$ there exists a monomial $w\neq u,v$ in $G(I)$ with $u\prec w$ and ${\rm supp}(w)\subseteq {\rm supp}(u)\cup {\rm supp}(v)$.
\end{prop}
\begin{proof}
Let $G(I)=\{u_1, \ldots, u_t\}$ and assume that $u_1< u_2 <\ldots < u_t$. We consider the linear order $u_1 \succ u_2\succ \ldots \succ u_t$ on $G(I)$. We prove that using this order, $I$ satisfies the desired property. So suppose that there exist $1\leq i< j \leq t$ such that ${\rm gcd}(u_i,u_j)=1$. We have to show that there exists $k\neq i, j$ and $k<j$ such that ${\rm supp}(u_k)\subseteq {\rm supp}(u_i)\cup {\rm supp}(u_j)$. Since $<$ is a monomial order, it follows that $u_i^s< u_j^s$. Since $I^s$ has linear quotient with respect to $<$, we conclude that there exists a monomial $w\in G(I^s)$ such that $w< u_j^s$ and $$\frac{w}{{\rm gcd}(u_j^s, w)}=x_{\ell}, \ \ \ \ \ \ \ \ \ \ \ \ \ \ (\dag)$$ for some $1\leq \ell\leq n$, and moreover $$x_{\ell} \mid \frac{u_i^s}{{\rm gcd}(u_j^s, u_i^s)}. \ \ \ \ \ \ \ \ \ \ \ \ \ \ (\ddag)$$Since $w\in I^s$, we can write $w=u_{k_1}\ldots u_{k_s}$ for some integers $1\leq k_1\leq k_2 \leq \ldots \leq k_s\leq t$. Since $w< u_j^s$, it follows that $u_{k_1}<u_j$ and thus $k_1<j$. To simplify the notation, we denote $k_1$ by $k$. It follows from $(\dag)$ that $${\rm supp}(u_k)\subseteq {\rm supp}(u_j)\cup \{x_{\ell}\}. \ \ \ \ \ \ \ \ \ \ \ \ \ \ (\ast)$$We consider two cases.

{\bf Case 1.}
$x_{\ell}\nmid u_k$ or $x_{\ell} \mid u_j$. In this case $${\rm supp}(u_k)\subseteq {\rm supp}(u_j)\subseteq {\rm supp}(u_j)\cup {\rm supp}(u_i),$$and it follows from the first inclusion that ${\rm gcd}(u_k,u_j)\neq 1$ and therefore $u_k\neq u_i$. Hence the assertion is true in this case.

{\bf case 2.} $x_{\ell}\mid u_k$ and $x_{\ell} \nmid u_j$. It follows from $\ddag$ that $x_{\ell}\in {\rm supp}(u_i)$. Therefore $${\rm supp}(u_k)\subseteq {\rm supp}(u_j)\cup {\rm supp}(u_i).$$Since $I$ contains no variable, $u_k\neq x_{\ell}$. On the other hand since $x_{\ell} \nmid u_j$, using $\dag$, we conclude that $x_{\ell}^2\nmid w$ and therefore $x_{\ell}^2\nmid u_k$. This shows that ${\rm supp}(u_k)\neq \{x_{\ell}\}$. This shows that $${\rm supp}(u_k)\cap {\rm supp}(u_j)\neq \emptyset.$$ Thus ${\rm gcd}(u_j, u_k)\neq 1$. Therefore $u_k\neq u_i$ and this completes the proof.
\end{proof}

Considering the class of square-free monomial ideals, we obtain the following result.

\begin{cor} \label{sstrong}
Let $I$ be a square-free monomial ideal which contains no variable. Assume that for some integer $s\geq 1$, the ideal $I^s$ has linear quotient with respect to a monomial order. Then $I$ satisfies strong gcd condition.
\end{cor}
\begin{proof}
One should only note that for every couple of square-free monomials $u, v$ with ${\rm gcd}(u, v)=1$, a square-free monomial $w$ divides $uv$ if and only if ${\rm supp}(w)\subseteq {\rm supp}(u)\cup {\rm supp}(v)$. Now the assertion follows from Proposition \ref{strong}.
\end{proof}

Let $I$ be a square-free monomial ideal which contains no variable. In Corollary \ref{sstrong} we proved that if for some integer $s\geq 1$, the ideal $I^s$ has linear quotient with respect to a suitable order, then $I$ itself satisfies strong gcd condition. In fact we have no example to show that the same assertion does not hold if $I^s$ has linear quotient with respect to an "arbitrary" order on $G(I^s)$. So we present the following conjecture.

\begin{conj} \label{conj1}
Let $I$ be a square-free monomial ideal which contains no variable. Assume that for some integer $s\geq 1$, the ideal $I^s$ has linear quotient. Then $I$ satisfies strong gcd condition.
\end{conj}

We note that by \cite[Proposition 6]{b}, Conjecture \ref{conj1} is true when $I$ itself has linear quotient.

For a monomial ideal $I$ in $S$ the ring $S/I$, is called {\it Golod}
if all Massey operations on the Koszul complex of $S/I$ with respect of $\mathrm{x}= x_1,\ldots, x_n$ vanish.
Golod \cite{g} showed that the vanishing of the Massey operations is equivalent to the equality case in the following coefficientwise inequality of power-series which was
first derived by Serre:
$$\sum_{i \geq 0} \dim_{\mathbb{K}} {\rm Tor}_i^{S/I} (\mathbb{K},\mathbb{K}) t^i  \leq \frac{(1+t)^n}{1-t \sum_{i\geq 1} \dim_{\mathbb{K}} {\rm Tor}_i^S(S/I,\mathbb{K}) t^i}$$

We refer the reader to \cite{a} and \cite{gl} for further information on the Golod property.

By \cite[Theorem 5.5]{bj}, we know that if a monomial ideal $I$ satisfies strong gcd condition, then $S/I$ is a Golod ring. Thus as a consequence of Corollary \ref{sstrong}, we obtain the following result.

\begin{thm} \label{golod}
Let $I$ be a square-free monomial ideal which contains no variable. Assume that for some integer $s\geq 1$, the ideal $I^s$ has linear quotient with respect to a monomial order. Then $S/I$ is a Golod ring.
\end{thm}

Let $I$ be a monomial ideal. J{\"o}llenbeck \cite[Lemma 7.4]{j} proves that if $S/I$ is a Golod ring, then $I$ satisfies gcd condition. In Corollary \ref{scomp} we proved that if $I$ is a square-free monomial ideal which contains no variable, such that $I^s$ is componentwise linear for some $s\geq 1$, then $I$ satisfies gcd condition. So it is natural to ask whether an stronger result is true (see Questin \ref{quest}).

\begin{ques} \label{quest}
Let $I$ be a square-free monomial ideal which contains no variable. Assume that $I^s$ is componentwise linear for some $s\geq 1$. Is $S/I$ a Golod ring?
\end{ques}

We note that by \cite[Theorem 4]{hrw} the answer of Question \ref{quest} is positive in the case of $s=1$. We also note that by \cite[Theorem 1.1]{sw}, for every integer $s\geq 2$ and every monomial ideal $I$, the ring $S/I^s$ is Golod (see also \cite{hh1}).


\section{Lower bounds for Betti numbers} \label{sec3}

If $I\subseteq S$ is a square-free monomial ideal, we
can identify the set of minimal monomial generators of $I$ with the edge set of a clutter, defined as follows.

\begin{dfn}
Let $V$ be a finite set. A {\it clutter} $\mathcal{C}$ with vertex set $V$ consists of a set of
subsets of $V$, called the {\it edges} of $\mathcal{C}$, with the property that no edge contains another. A clutter $\mathcal{C}$ is called $k$-uniform if the cardinality of every edge of $\mathcal{C}$ is equal to $k$. Thus a clutter $\mathcal{C}$ is $2$-uniform if and only if it is a graph.
\end{dfn}

We write $V(\mathcal{C})$ to denote the vertices of $\mathcal{C}$, and $E(\mathcal{C})$ to denote its edges. Let $\mathcal{C}$ be a clutter and assume that $\mid V(\mathcal{C})\mid=n$. For every subset $e\subseteq \{1, 2, \ldots, n\}$, we write $x_e$ to denote the square-free monomial $\prod_{i \in e} x_i$. Then the {\it edge ideal} of $\mathcal{C}$ is defined to be
$$I(\mathcal{C}) = (x_e : e\in E(\mathcal{C})),$$as an ideal in the polynomial ring $S = \mathbb{K}[x_1, \dots, x_n]$.

\begin{dfn}
Let $\mathcal{C}$ be a clutter. A subset $\{e_1,\ldots, e_t\} \subseteq E(\mathcal{C})$ is called an {\it induced matching} of $\mathcal{C}$, if for every $i\neq j$, $e_i\cap e_j = \emptyset$ and moreover $e_1, \ldots, e_t$ are the only edges of $\mathcal{C}$ which are contained in $\bigcup_{i=1}^t e_i$. The maximum cardinality of an induced matching of $\mathcal{C}$ is called the {\it induced matching number} of $\mathcal{C}$ and is denoted by ${\rm Indmath}(\mathcal{C})$.
\end{dfn}

Let $\mathcal{C}$ be a $k$-uniform clutter and $s\geq 2$ be an integer. In the following theorem we provide a lower bound for $\beta_{1,ks+k}(I(\mathcal{C})^s)$ and $\beta_{2,ks+2k}(I(\mathcal{C})^s)$ in terms of ${\rm Indmath}(\mathcal{C})$.

\begin{thm} \label{beti}
Let $\mathcal{C}$ be a $k$-uniform clutter and assume that ${\rm Indmath}(\mathcal{C})=t$. Then for every integer $s\geq 2$ the following inequalities hold. $$\beta_{1,ks+k}(I(\mathcal{C})^s)\geq 2\binom{t}{2} \ \ \ \ \ \ \ \ \ \beta_{2,ks+2k}(I(\mathcal{C})^s)\geq 3\binom{t}{3}$$
\end{thm}
\begin{proof}
Set $I=I(\mathcal{C})$. By assumption there exist monomials $u_1, \ldots, u_t\in G(I)$ such that for every $1\leq i<j\leq t$ we have ${\rm gcd}(u_i, u_j)=1$ and $u_1, \ldots, u_t$ are the only monomials in $G(I)$ which divide $\prod_{i=1}^tu_i$. To prove the first inequality, for every $1\leq i<j\leq t$, consider the open intervals $(1,u_i^su_j)$ and $(1,u_iu_j^s)$ in the lcm lattice of $I^s$. The chain complex of the first interval consists of two isolated vertices $u_i^s$ and $u_i^{s-1}u_j$. So using by \cite[Theorem 3.3]{gpw} we conclude that $$\beta_{1,u_i^su_j}(I^s)={\rm dim}_{\mathbb{K}}\tilde{H}_0(\Delta(1,u_i^su_j);\mathbb{K})=1.$$ Similarly $\beta_{1,(u_iu_j^s)}(I^s)=1$. Since ${\rm deg}(u_i^su_j)={\rm deg}(u_iu_j^s)=ks+k$ and since we have $\binom{t}{2}$ many choices for $i$ and $j$, we conclude that $$\beta_{1,ks+k}(I^s)\geq 2\binom{t}{2}.$$

We prove the second inequality using a similar argument. For every $1\leq i<j< l\leq t$, consider the open intervals $(1,u_i^su_ju_l)$ and $(1,u_iu_j^su_l)$ and $(1,u_iu_ju_l^s)$ in the lcm lattice of $I^s$. The dimension of the chain complex of each interval is equal to one and one can easily check that the first reduced homology of the chain complex of each interval is a one dimensional $\mathbb{K}$-vector space. This shows that $$\beta_{2,(u_i^su_ju_l)}(I^s)=\beta_{2,(u_iu_j^su_l)}(I^s)=\beta_{2,(u_iu_ju_l^s)}(I^s)=1.$$Since ${\rm deg}(u_i^su_ju_l)={\rm deg}(u_iu_j^su_l)={\rm deg}(u_iu_ju_l^s)=ks+2k$ and since we have $\binom{t}{3}$ many choices for $i,j$ and $l$, we conclude that $$\beta_{2,ks+2k}(I^s)\geq 3\binom{t}{3}.$$
\end{proof}

We recall that the {\it Castelnuovo-Mumford regularity} of an ideal $I$ is equal to
$${\rm reg}(I) = {\rm max}\{j: \beta_{i, i+j}(I)\neq 0 \ \ {\rm for \ \ some} \ \ i\}.$$

As an immediate consequence of Theorem \ref{beti}, we conclude the following result.

\begin{cor}
Let $\mathcal{C}$ be a $k$-uniform clutter and $s\geq 2$ be an integer.
\begin{itemize}
\item[(i)] If ${\rm Indmath}(\mathcal{C})=2$, then ${\rm reg}(I(\mathcal{C})^s)\geq ks+k-1$.
\item[(ii)] If ${\rm Indmath}(\mathcal{C})\geq 3$, then ${\rm reg}(I(\mathcal{C})^s)\geq ks+2k-2$.
\end{itemize}
\end{cor}






\begin{thebibliography}{10}

\bibitem{a} L.L. Avramov, Infinite free resolutions, in: J. Elias et al.,
  Six lectures on commutative algebra. {\it Prog. Math.} {\bf 166}, 1--118, Basel: Birkh{\"a}user, 1998.

\bibitem {b} A. Berglund, Shellability and the strong gcd-condition, {\it Electron. J. Combin.}, {\bf 16 (2)}, (2009), 1--7.

\bibitem {bj} A. Berglund, M. J{\"o}llenbeck, On the Golod property of Stanley--Reisner rings, {\it J. Algebra.} {\bf 315} (2007) 249--273.

\bibitem {f} R. Fr${\rm \ddot{o}}$berg, On Stanley-Reisner rings, in: Topics in algebra, {\it Banach Center Publications}, {\bf 26}
Part 2, (1990), 57--70.

\bibitem {g} E.S. Golod, On the homology of some local rings, Soviet Math. Dokl. {\bf 3} (1962) 745--748.

\bibitem{gl} T. H. Gulliksen, G. Levin, Homology of local rings, Queens Papers in Pure and Appl.
  Math. {\bf 20}, Kingston, Ontario: Queens University, 1969.

\bibitem {gpw} V. Gasharov, I. Peeva, V. Welker, The lcm-lattice in
    monomial resolutions, {\it Math. Res. Lett.} {\bf 6} (1999), no. 5-6,
    521--532.

\bibitem {hh} J. Herzog, T. Hibi, {\it Monomial Ideals}, Springer-Verlag,
    2011.
    
\bibitem {hh2} J. Herzog, T. Hibi, Componentwise linear ideals, {\it Nagoya Math. J.} {\bf 153} (1999), 141--153  

\bibitem {hhz} J. Herzog, T. Hibi, X. Zheng, Monomial ideals whose powers
    have a linear resolution, {\it Math. Scand.} {\bf 95} (2004), no. 1,
    23--32.
    
\bibitem {hh1} J. Herzog, C. Huneke, Ordinary and symbolic powers are Golod, {\it Adv. Math.} {\bf 246} (2013) 89--99.
 
\bibitem {hrw} J. Herzog, V. Reiner, V. Welker, Componentwise linear ideals and Golod
rings, {\it Michigan Math. J.} {\bf 46} (1999), 211--223.

\bibitem {ht} J. Herzog, Y. Takayama, Resolutions by mapping cones, in: The Roos Festschrift volume
Nr.2(2), {\it Homology, Homotopy and Applications} {\bf 4}, (2002), 277--294.

\bibitem {j} M. J${\rm \ddot{o}}$llenbeck, On the multigraded Hilbert- and Poincare series of monomial rings, {\it J. Pure Appl.
Algebra.} {\bf 207}, (2006), 261--298.

\bibitem {np} E. Nevo, I. Peeva, $C_{4}$-free edge ideals, {\it J. Algebr. Comb.} {\bf 37} (2013), 243--248.

\bibitem {sw} S. A. Seyed Fakhari, V. Welker, The Golod property for products and high symbolic powers of monomial ideals , Preprint.

\bibitem {s} B. Sturmfels, Four counterexamples in combinatorial algebraic geometry, {\it J. Algebra} {\bf 230}
(2000), 282--294.
\end{thebibliography}
\end{document}